\documentclass{amsart}
\usepackage{amsmath,amssymb,amsthm, mathrsfs}
\usepackage{boxedminipage}
\usepackage{graphicx}
\usepackage[margin=1in]{geometry}
\usepackage{mathpazo}
\usepackage{fancyhdr}
\usepackage[utf8]{inputenc}
\usepackage{color}
\usepackage{tikz}

\newtheorem{innercustomthm}{Theorem}
\newenvironment{customthm}[1]
  {\renewcommand\theinnercustomthm{#1}\innercustomthm}
  {\endinnercustomthm}
\newtheorem{theorem}{Theorem}[section]
\newtheorem{lemma}[theorem]{Lemma}
\newtheorem{conjecture}[theorem]{Conjecture}

\newtheorem{varexample}[theorem]{Example}

\newenvironment{example}{\begin{varexample}
\begin{normalfont}}{\end{normalfont}
\end{varexample}}

\theoremstyle{definition}

\newtheorem{definition}[theorem]{Definition}

\definecolor{gold}{rgb}{1.0,0.75,0.0}

\title{The Gonality of Rook Graphs}

\author{Noah Speeter}

\begin{document}

\maketitle

\begin{abstract}
Two dimensional rook graphs are the Cartesian product of two complete graphs. In this paper we prove that the gonality of these graphs is the expected value of $(n-1)m$ where $n$ is the size of the smaller complete graph and $m$ is the size of the larger. furthermore we compute the 2 and 3 gonalities of these graphs. We also explore the scramble number of these graphs, which is a new graph invariant and a lower bound on the gonality.
\end{abstract}

\section{Introduction}
 
In \cite{BakerNorine} Baker and Norine describe the divisor theory of graphs, which is a discrete analogue for the divisor theory of algebraic curves. Divisor theory of graphs is often described as a chip firing game in which a player is given a graph which some number of chips on each vertex, and the player can fire a vertex by transferring one chip away from the vertex along each of its edges. The player "wins" the game if they can get to a configuration where no vertex has a negative number of chips. The gonality gon$(G)$ of a graph $G$ is a graph invariant that tells us the fewest number of chips the player would require such that one chip can be stolen and the player could still win.

The primary motivation for computing the gonality of graphs is to better understand the gonality of algebraic curves, which is the minimal degree of a nonconstant rational map from that curve to the projective line. 
Given some algebraic curve $C$, we can degenerate it to a union of lines $C'$ and then produce the dual graph of $C'$ which we will call $G$. Then the minimum gonality over all refinements of $G$ is a lower bound for the gonality of $C$. 

In this paper, we compute the gonality of two dimensional rook graphs. One motivation is that rook graphs are the dual graphs of a certain degeneration of complete intersection curves. In \cite{TW} van Dobben de Bruyn and Gijswijt raise the question of computing the gonality of $n-$dimensional cubes $Q_n$, which are examples of rook graphs. In \cite{Ralph} Aidun and Morrison show that gon$(K_n\square K_m)=(n-1)m$ if $n\leq m$ and $n\leq 5$. The main result of this paper is the full generalization of this theorem.

\begin{theorem}
\label{gonality}
If $n\leq m$, $\textrm{gon}(K_n\square K_m)=(n-1)m$.
\end{theorem}

This result matches the lower bound for the gonality of complete intersection curves given by Lazarsfeld in \cite{lazarsfeld1994lectures}, where Exercise 4.12 shows the complete intersection of hypersurfaces of degrees $2\leq a_1\leq a_2\leq\cdots\leq a_{r-1}$ has gonality $d\geq (a_1-1)a_2\dots a_{r-1}$.
While Theorem \ref{gonality} is a significant result, it does not preclude the possibility of some refinement of a rook graph having smaller gonality. However in \cite{paper1} we introduce the graph invariant known as the scramble number $sn(G)$ which acts as a lower bound for gonality and has the nice property of being invariant under refinement. While there are many rook graphs whose scramble number is strictly less than the gonality, we also prove the following.

\begin{theorem}
If $m\geq(n-2)(n-1)$, $sn(K_n\square K_m)=(n-1)m$.
\label{scramble}
\end{theorem}

This shows that, at least in cases where $m$ is sufficiently larger than $n$, the gonality of $K_n\square K_m$ is invariant under all refinements. 

In this paper, we also compute the scramble number of some three dimensional rook graphs. In section 6, we prove the following result.

\begin{theorem}
Let $2\leq n\leq m$, then $sn(K_2\square K_n\square K_m)=nm$.
\end{theorem}
The gonality of this family of rook graphs was known to be at most $nm$, and therefore this result also computes the gonality of all three dimensional rook graphs where the smallest dimension is 2. Further study is required to determine if any other three dimensional rook graphs have a scramble number matching its gonality.

In section 7, we explore higher gonalities of two dimensional rook graphs. There are not many families of graphs in which higher gonalities are currently known. However, for two dimensional rook graphs, we prove the following.

\begin{theorem}
Given the rook graph $K_n\square K_m$, gon$_2(K_n\square K_m)=nm-1$ and gon$_3(K_n\square K_m)=nm$.
\end{theorem}

It is not currently known if the 4 gonality behaves nicely for these graphs as it does with the 2 and 3 gonalities.

\section{Preliminaries}
We begin by establishing terminology and giving background results.

\subsection{Graphs}
For the entirety of this paper, we will assume that our graphs are connected and without loops. However, multiple edges between vertices are allowed. Given a graph $G$, we denote the vertex set by $V(G)$ and the edge set by $E(G)$. If $A\subseteq V(G)$ then the complement of $A$ will be denoted as $A^c$. 
\begin{definition}
A partition of the vertices into two sets, $(A, A^c)$ is referred to as a \emph{cut}. The \emph{cut-set} $E(A,A^c)$ is the set of edges that have one end in $A$ and the other end in $A^c$. 
\end{definition}

\begin{definition}
Given two graphs $G$ and $H$, we can construct a new graph by taking their \emph{Cartesian product} $G\square H$, with vertex set $$V(G\square H)=\{(x,y)|x\in V(G), y\in V(H)\}$$ and edge set $$E(G\square H)=\{(x,y_1)\sim (x,y_2)|y_1\sim y_2\in E(H)\}\cup \{(x_1,y)\sim (x_2,y)|x_1\sim x_2\in E(G)\}.$$
\end{definition}

\begin{definition}
A \emph{rook graph} is the Cartesian product of 2 or more complete graphs. \end{definition}  

The $n\times m$ rook graph $K_n\square K_m$ can be represented as an $n\times m$ lattice where two lattice points are adjacent if they share either the same row or column. The name rook graph comes from this lattice representation because two vertices are adjacent if they are a rooks move apart.

\subsection{Graph Divisors and Chip Firing}
\begin{definition} A \emph{divisor} $D$ on a graph $G$ is a $\mathbb{Z}-$linear combination of the vertices in $G$, or alternatively, an integer vector in  $\mathbb{Z}^{V(G)}$.\end{definition}

Divisors on a graph are often described as stacks of poker chips on each vertex, where a negative number on a vertex is thought of as a "debt". Because of this, divisors, are sometimes referred to as chip configurations.
\begin{definition}
The \emph{degree} of a divisor, denoted deg$(D)$, is the sum of all coordinates of the vector $D\in \mathbb{Z}^{V(G)}$, or simply the sum of all chips and debts.\end{definition}
We "fire" a vertex $v$ by transferring one chip along each edge connected to $v$, away from that fired vertex. If we obtain $D'$ by starting with divisor $D$ and firing some vertex subset $A$, then we can obtain $D$ from $D'$ by firing $A^c$. This imposes equivalence classes on the set of divisors of a fixed degree, where two divisors are equivalent if and only if there are some series of chip fires apart from one another.

\begin{definition} The following terms are needed to understand Dhar's burning algorithm and the proof of Theorem \ref{gonality}
\begin{itemize}

\item A divisor is \emph{effective} if all vertices have a non-negative number of chips, and we say a divisor is \emph{effective away from $v$} if all vertices, apart from $v\in V(G)$, have a non-negative number of chips.

\item A divisor is \emph{v-reduced} if it is effective away from $v$, and firing any subset of $V(G)\setminus v$ results in a divisor that is not effective away from $v$.

\item A divisor $D$ has \emph{rank} of at least $r$ if, for every effective divisor $E$ of degree $r$, $D-E$ is equivalent to an effective divisor.

\item The \emph{gonality} of a graph $G$, denoted gon$(G)$, is the fewest number of chips needed to construct a divisor of rank $1$.
\end{itemize}
\end{definition}

Another way to understand the rank of a divisor is in the context of a chip firing game. If a divisor $D$ has rank $r$, that means  if someone were to steal $r$ chips from anywhere on the graph, even from vertices that have 0 or a debt of chips, then there is some series of chip firings one could perform to get back an effective divisor. It also means that there is at least one way to steal $r+1$ chips that would make getting back to an effective divisor impossible.

\begin{lemma}\label{rank} If a divisor $D$ is $v$-reduced, and $v$ has 0 or fewer chips, then $D$ does not have positive rank.\end{lemma}
\begin{proof}
Let the ''stolen" chip be from the specified vertex $v$. Then $v$ has a negative number of chips and since $D$ is $v-$reduced, firing any other set of vertices will result in a non-effective divisor.
\end{proof}

\subsection{Dhar's burning algorithm}
Here we briefly review Dhar's burning algorithm, which was first introduced in \cite{DBA}. This algorithm tells us what series of chip fires need to occur in order to get a divisor $D$ to an equivalent divisor $D'$ that is $v$-reduced for some chosen vertex $v$. We begin the algorithm by starting a "fire" at our chosen vertex. it should be noted by the reader that this notion of starting a fire is separate from chip firing. The fire spreads to all edges adjacent to a burning vertex, and a vertex will burn if it has more burning adjacent edges than it does chips. If the entire graph burns, then the divisor is reduced at the vertex $v$. Otherwise, if there is some set of unburnt vertices left, we fire these vertices and start a new fire at $v$. The process continues until we have a divisor in which all vertices will burn.

\begin{lemma}\label{poorest row}
If $D$ is an effective divisor of degree at most $n-2$ on $K_n$, then a fire started on any vertex with no chips will burn the entire graph.
\end{lemma}

\begin{proof}
Let $k$ be the number of unburnt vertices left after starting a fire on some vertex with no chips. For these $k$ vertices to not catch fire, they each must have a minimum of $n-k$ chips. Since our divisor is effective and has degree at most $n-2$, we must have $k(n-k)\leq (n-2)$. The only values for which this inequality holds are $k=0$ and $k=n$. We know $k\neq n$ because one vertex was burnt at the beginning. Therefore all vertices must burn. 
\end{proof}

\section{the gonality of $K_n\square K_m$}
For the remainder of the paper, we will assume without loss of generality that $n\leq m$.

\begin{definition}Given a divisor $D$ on $K_n\square K_m$, a \emph{poorest column} of $D$ refers to a copy of $K_n$ that contains the fewest number of chips. Similarly, a \emph{poorest row} refers to a copy of $K_m$ that contains the fewest number of chips. \end{definition}

Note that a poorest row/column need not be unique.

\begin{customthm}{1.1}\label{gonality-restate}
If $n\leq m$, $\textrm{gon}(K_n\square K_m)=(n-1)m$.
\end{customthm}

\begin{proof}First, we observe that $\textrm{gon}(K_n\square K_m)\leq(n-1)m$ since the divisor that has one chip on every vertex except for one row which is left empty, has degree $(n-1)m$ and positive rank. We then need to show that every divisor of degree $(n-1)m-1$ does not have positive rank.

Let $D$ be an effective divisor of degree $(n-1)m-1$, and without loss of generality, we assume that $D$ has a maximal number of chips in its poorest column among all equivalent effective divisors. A poorest column contains at most $n-2$ chips and by lemma \ref{poorest row}, starting a fire at one of the vertices with no chips will make the entire column burn. Additionally, the poorest row of $D$ divisor will have at most $m-2$ chips since $(n-1)m-1\leq n(m-1)-1$. Because an entire column has burned, the poorest row must also burn. We then assume for contradiction that there is some nonempty subset $U\subset V(K_n\square K_m)$ which is left unburnt. We can then fire $U$ to obtain a new effective divisor $D'$.

If column $k$ has $j>0$ vertices in $U$, then firing $U$ will transfer $j(n-j)$ chips from the $j$ vertices in $U$ to the other $n-j$ vertices in column $k$. Since $D'$ is effective, column $k$ will have at least $j(n-j)$ chips in $D'$. Because the poorest row has burned, $j< n$ and thus, $j(n-j)\geq n-1$. Therefore any column that intersects with $U$, will not become the poorest column of $D'$. However, every column that doesn't intersect with $U$ will gain $|U|$ chips once $U$ is fired. This is a contradiction since we assumed $D$ maximized the number of chips in the poorest column. Therefore all vertices of $K_n\square K_m$ must burn, meaning $D$ is $v$-reduced for some vertex that has 0 chips. Then by Lemma \ref{rank}, we conclude that $D$ does not have positive rank.  
\end{proof}

\section{the scramble number of a graph}
\definition{Given a graph $G$, we say $S$ is a \emph{scramble} on $G$ if $S$ is a collection of non-empty vertex subsets of $G$ such that every vertex subset is connected.}

A scramble is a generalization of a bramble, with the removed condition that every pair of vertex subsets need to be connected. We will refer to the vertex subsets of a scramble as \emph{eggs}.
\begin{example}
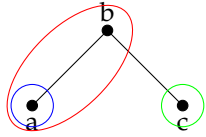
\begin{figure}[ht]
\begin{tikzpicture}
\filldraw (0,0) circle (2pt);
\filldraw (1,1) circle (2pt);
\filldraw (2,0) circle (2pt);
\node at (0,-0.25) {a};
\node at (1,1.25) {b};
\node at (2,-0.25) {c};
\draw (0,0)--(1,1);
\draw (1,1)--(2,0);
\draw[rotate around={45:(0.5,0.5)},red] (0.5,0.5) ellipse (30pt and 15pt);
\draw [blue](0,0) circle (8pt);
\draw [green](2,0) circle (8pt);

\end{tikzpicture}
\caption{A graph with a scramble consisting of the eggs $\{a\}, \{a,b\}, \{c\}$}
\label{Fig:scramble1}
\end{figure}
\end{example}
In Figure \ref{Fig:scramble1}, We have a scramble on our graph that contains three eggs. Notice that two eggs can share vertices and one egg can even be entirely contained in another. 
\begin{definition}given a scramble $S$ on graph the $G$, a set $H\subseteq V(G)$ is a \emph{hitting set} of $S$ if $H\cap E\neq \emptyset$ for every egg $E$ in $S$.\end{definition}

The set $\{a,b,c\}$ is a hitting set of the scramble shown in Figure \ref{Fig:scramble1}, However $b$ is not needed, since $\{a,c\}$ intersects non-trivially with all three eggs. Since there is no single vertex that intersects with all eggs, $\{a,c\}$ is a minimal hitting set.

\begin{definition}We say that $(X,Y)$ is an \emph{egg cut} of our scramble $S$ if $X\coprod Y=V(G)$ and there are two eggs $E_1, E_2$ such that $E_1\subset X, E_2\subset Y$. The set of edges $E(X,Y)$ is referred to as the \emph{egg cut set} or simply the \emph{cut set}, when the context is clear that we are talking about an egg cut.\end{definition}

In Figure \ref{Fig:scramble1} both $\{a,b\}\coprod\{c\} $ and $\{a\}\coprod \{b,c\}$ are egg cuts with a cutset of size 1. $\{b\}\coprod\{a,c\}$ is not an egg cut because the subset $\{b\}$ does not contain an egg. 

\begin{definition}Given a scramble $S$ with minimal hitting set $H$ and minimal egg cut set $E(X,Y)$, the \emph{order} of $S$, denoted $||S||$ is min$\{|H|,|E(X,Y)|\}$.\end{definition}

\begin{definition}The \emph{scramble number} of a graph, denoted $sn(G)$ is the largest order of a scramble that exists on the graph.\end{definition}

\begin{theorem}\label{scramble_bound} \cite{paper1}
For any graph $G$, we have $sn(G)\leq \textrm{gon}(G)$.
\end{theorem}

\begin{example}
In Figure \ref{Fig:Refinement} we see two graphs, where the left graph is a refinement of the graph on the right. Both graphs have a scramble number of 2. The graph on the right has gonality 2 while the graph on the left has gonality 3. A more detailed explanation on how to compute the scramble number and gonality of these graphs can be found in \cite{paper1}.
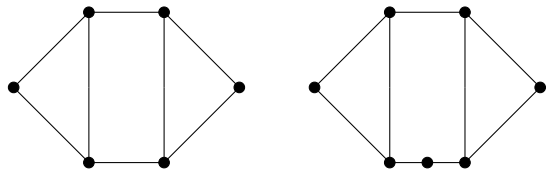
\begin{figure}[ht]
\begin{tikzpicture}
\filldraw (0,1) circle (2pt);
\filldraw (1,0) circle (2pt);
\filldraw (1,2) circle (2pt);
\filldraw (2,0) circle (2pt);
\filldraw (2,2) circle (2pt);
\filldraw (3,1) circle (2pt);
\draw (0,1)--(1,0);
\draw (0,1)--(1,2);
\draw (1,0)--(1,1);
\draw (1,0)--(2,0);
\draw (1,1)--(1,2);
\draw (1,2)--(2,2);
\draw (2,0)--(2,1);
\draw (2,0)--(3,1);
\draw (2,1)--(2,2);
\draw (2,2)--(3,1);

\filldraw (4,1) circle (2pt);
\filldraw (5,0) circle (2pt);
\filldraw (5,2) circle (2pt);
\filldraw (6,0) circle (2pt);
\filldraw (6,2) circle (2pt);
\filldraw (7,1) circle (2pt);
\filldraw (5.5,0) circle (2pt);
\draw (4,1)--(5,0);
\draw (4,1)--(5,2);
\draw (5,0)--(5,1);
\draw (5,0)--(6,0);
\draw (5,1)--(5,2);
\draw (5,2)--(6,2);
\draw (6,0)--(6,1);
\draw (6,0)--(7,1);
\draw (6,1)--(6,2);
\draw (6,2)--(7,1);
\end{tikzpicture}
\caption{Two graphs with the same scramble number, but different gonalities}
\label{Fig:Refinement}
\end{figure}
\end{example}

\section{computing the scramble number of rook graphs}

We begin this section with the following theorem which will assist us in computing the scramble number of many rook graphs.
\begin{theorem}
\label{cutset}
 Given a cut $A\coprod B$ on $K_n\square K_m$ Such that $|A|, |B|\geq n-1$, we have $|E(A,B)|\geq (n-1)m$. 
\end{theorem}
\begin{proof}
First, we establish that the proposition holds if either $A$ or $B$ has at most $m$ vertices. Assume without loss of generality that $n-1\leq |A|\leq m$. Every vertex in $K_n\square K_m$ has degree $n+m-2$, and thus $$|E(A,B)|=(n+m-2)|A|-2k,$$ where $k$ is the number of edges with both ends in $A$. The number $k$ reaches a maximal value of $\binom{|A|}{2}$ when all vertices of $A$ are in a single row or column. Thus we need only to consider the case where $A$ is contained in a single row. Then $$|E(A,B)|=|A|(m-|A|)+|A|(n-1),$$ where $|A|(m-|A|)$ represents the number of horizontal cut edges between two vertices in the same row, and $|A|(n-1)$ represents the number of vertical cut edges between two vertices contained in the same column. $$|A|(m-|A|)+|A|(n-1)=|A|(m+n-|A|-1).$$ This product is minimized at the boundary cases when $|A|=n-1$ or $|A|=m$, and the resulting product is $(n-1)m$.

Next, we establish that the proposition holds if both $A$ and $B$ contain an entire row. Without loss of generality, assume the first row is entirely in $A$ and the second row is entirely in $B$. Then there are $m$ cut edges between those two rows and each of the remaining $(n-2)m$ vertices will have a vertical cut edge between itself and one of the first two rows. Thus $|E(A,B)|\geq (n-1)m$.

Finally, we assume that both $|A|, |B|\geq m+1$, there are $0\leq i\leq n-2$ rows completely contained in $A$ and no rows completely contained in $B$. Because $|B|\geq m+1$, there will be at least $i(m+1)$ vertical cut edges between the vertices in $B$ and the $i$ rows in $A$. The remaining $n-i$ rows contain at least one vertex in $A$ and one vertex in $B$. Thus each of these rows contains at least $m-1$ horizontal cut edges. Thus $$|E(A,B)|\geq i(m+1)+ (n-i)(m-1)> i(m-1)+(n-i)(m-1)=n(m-1)\geq (n-1)m.$$  
\end{proof}

\begin{lemma} For all $m\geq 2$ we have $sn(K_2\square K_m)=m=\textrm{gon}(K_2\square K_m)$
\end{lemma}

\begin{proof}
By Theorems \ref{gonality} and \ref{scramble_bound} it suffices to find a scramble of order $n$. Let $S$ be the scramble where every vertex is its own egg. Then the minimal hitting set is all of $V(K_2\square K_m)$. By Theorem \ref{cutset}, any egg cut of $S$ will have a cut set of size greater than or equal to $m$. Therefore $||S||=m$. 
\end{proof}
\begin{lemma} For all $m\geq 3$, we have $sn(K_3\square K_m)=2m=\textrm{gon}(K_3\square K_m)$. \end{lemma}

\begin{proof}
Again by Theorems \ref{gonality} and \ref{scramble_bound}, it suffices to find a scramble of order $2n$. Let $S$ be the scramble where the eggs consist of every two adjacent vertices. By Theorem \ref{cutset}, The minimal egg cutset size is $2m$. If $H$ is a hitting set of $S$, then $H$ must contain 2 out of 3 vertices in every column. Then $|H|\geq 2m$ and thus $||S||=2m.$
\end{proof}

Unfortunately, for larger values of $n$, the scramble number of $K_n\square K_m$ does not always match the gonality, and it becomes increasingly more difficult to compute. A good candidate for a maximal scramble on a rook graph is the scramble where the eggs are all connected subsets of size $n-1$. We will refer to this scramble as $S^*_{n,m}$. By Theorem \ref{cutset}, the minimal cutset of $S^*_{n,m}$ is always $(n-1)m=\textrm{gon}(K_n\square K_m)$, so the order of $S^*_{n,m}$ is only less than the gonality if the minimal hitting set is too small. We can find a minimal hitting set indirectly by instead looking for a \emph{maximal avoidance set}, which is the largest subset of vertices that do not contain an entire egg. In the case of $S^*_{n,m}$, we need to avoid all connected subsets of size $n-1$. The complement of a maximal avoidance set is a minimal hitting set. Thus, if a scramble on $K_n\square K_m$ has a maximal avoidance set of size $k$, then it has a minimal hitting set of size $nm-k$.

\begin{example}
We show that $sn(K_4\square K_4)=11.$ Note that this is strictly smaller than gon$(K_4\square K_4)=12.$ 

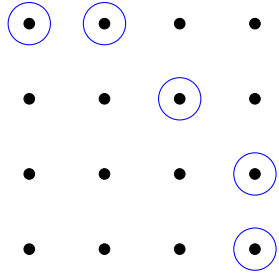
\begin{figure}[ht]
\begin{tikzpicture}
\filldraw (0,0) circle (2pt);
\filldraw (0,1) circle (2pt);
\filldraw (0,2) circle (2pt);
\filldraw (0,3) circle (2pt);
\filldraw (1,0) circle (2pt);
\filldraw (1,1) circle (2pt);
\filldraw (1,2) circle (2pt);
\filldraw (1,3) circle (2pt);
\filldraw (2,0) circle (2pt);
\filldraw (2,1) circle (2pt);
\filldraw (2,2) circle (2pt);
\filldraw (2,3) circle (2pt);
\filldraw (3,0) circle (2pt);
\filldraw (3,1) circle (2pt);
\filldraw (3,2) circle (2pt);
\filldraw (3,3) circle (2pt);
\draw [blue](0,3) circle (8pt);
\draw [blue](1,3) circle (8pt);
\draw [blue](2,2) circle (8pt);
\draw [blue](3,1) circle (8pt);
\draw [blue](3,0) circle (8pt);

\end{tikzpicture}
\caption{A maximal avoidance set of size 5}
\label{4x4}
\end{figure}
In Figure \ref{4x4} we see that the 5 circled vertices do not contain a connected subset of 3 or more vertices. However, any collection of 6 vertices must contain a connected subset of 3 vertices. This is because selecting 6 vertices from 4 columns either requires selecting 3 or more vertices from a single column or selecting 2 vertices from 2 separate columns. If the latter of these cases occurs, then either the 4 selected vertices are already  connected or every row has one vertex and the next vertex selected will result in a connected subset of 3 vertices.

Therefore the 11 vertices not circled in figure \ref{4x4} forms a minimal hitting set for the scramble $S^*_{4,4}$, meaning $||S^*_{4,4}||=11$. We then conclude that $sn(K_n\square K_4)=11$ because any scramble on this graph with a minimal hitting set larger than 11 would require eggs of size 2 or 1, and any such scramble would have a minimum egg cut set that is size 10 or smaller.
\end{example}

\begin{customthm}{1.2} If $m\geq (n-2)(n-1)$, then $sn(K_n\square K_m)=||S^*_{n,m}||=(n-1)m$.\end{customthm}
\begin{proof}
Recall that $S^*_{n,m}$ is the scramble where every connected vertex subset of size $n-1$ is an egg. We show that any set $A$ with $m+1$ vertices cannot be an avoidance set. Given such a set, there must be one column with at least 2 vertices in $A$. Without loss of generality, we say $(1,1)$ and $(2,1)$ are in $A$. Since $(1,1)\sim (2,1)$ all of the vertices in $A$ that are in the first two rows will form a connected component, and therefore $A$ can have at most $n-2$ vertices in the first two columns. This means that $A$ has at least $m+1-(n-2)$ vertices in the remaining $n-2$ rows. Since $m\geq (n-2)(n-1)$, we have $$m+1-(n-2)\geq(n-2)(n-1)-(n-2)+1= (n-2)^2+1,$$ and therefore $A$ must have at least $n-1$ vertices in a single column, meaning that $A$ cannot be an avoidance set. We conclude that any hitting set of $S^*_{n,m}$ must be of size $(n-1)m$ or larger.
\end{proof}
\begin{figure}[ht]
\begin{tikzpicture}
\filldraw (0,0) circle (2pt);
\filldraw (0,1) circle (2pt);
\filldraw (0,2) circle (2pt);
\filldraw (0,3) circle (2pt);
\filldraw (1,0) circle (2pt);
\filldraw (1,1) circle (2pt);
\filldraw (1,2) circle (2pt);
\filldraw (1,3) circle (2pt);
\filldraw (2,0) circle (2pt);
\filldraw (2,1) circle (2pt);
\filldraw (2,2) circle (2pt);
\filldraw (2,3) circle (2pt);
\filldraw (3,0) circle (2pt);
\filldraw (3,1) circle (2pt);
\filldraw (3,2) circle (2pt);
\filldraw (3,3) circle (2pt);
\filldraw (4,0) circle (2pt);
\filldraw (4,1) circle (2pt);
\filldraw (4,2) circle (2pt);
\filldraw (4,3) circle (2pt);
\draw [blue](0,3) circle (8pt);
\draw [blue](1,3) circle (8pt);
\draw [blue](2,2) circle (8pt);
\draw [blue](3,2) circle (8pt);
\draw [blue](4,1) circle (8pt);
\draw [blue](4,0) circle (8pt);

\end{tikzpicture}
\caption{An avoidance set of size 6=$m+1$}
\label{4x5}
\end{figure}
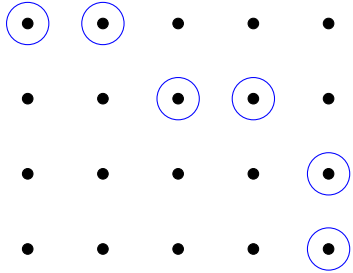
\begin{theorem}
If $m<(n-2)(n-1)$, then $||S^*_{n,m}||< (n-1)m$.
\end{theorem}
\begin{proof}
First, we consider the case where $m= k(n-2)+r$ for some nonzero remainder $1\leq r<n-2$. Then we aim to construct an avoidance set of size $m+1$. In the first $k$ rows, we select $n-2$ vertices per row such that each vertex comes from a unique column. Since $m=k(n-2)+r<(n-2)(n-1)$ we know $k\leq n-2$, and thus there are $r$ columns and at least 2 rows that have no vertices selected. Selecting any set of $r+1$ vertices from these leftover rows and columns will result in an avoidance set of size $m+1$ (see Figure \ref{4x5} for the case $n=4, m=5, k=2, r=1)$.

Next, we consider the case where $m=k(n-2)$ for some $k\leq n-2$. We select $n-2$ vertices from each of the first $k-1$ rows, and $n-3$ vertices from row $k$ such that each vertex is from a unique column. From the last remaining column that has not yet had a vertex selected, we select two from rows $k+1,\dots, n$. The selected vertices then form an avoidance set of size $m+1$. 

 \end{proof}

In cases where $m$ and $n$ are relatively close in size, $S^*_{n,m}$ might not be optimal. That is, we can find a scramble of larger order on that graph. This makes computing the scramble number of general rook graphs increasingly difficult as $m$ and $n$ get large.

\begin{example}
Given the graph $K_6\square K_6$, we first consider the scramble $S^*_{6,6}$ consisting of all possible eggs of size 5. This scramble has a hitting set of size 24 since we can construct an avoidance set of size 12 pictured in Figure \ref{6x6}. However, if we create a new scramble $T^*$ by augmenting the egg set of $S^*_{6,6}$ to also include all 4-vertex squares, we increase the minimum hitting set to size 27 without decreasing the minimum cut set size. One could also include ''S" and ''Z" shaped 4-vertex subsets into the egg set without diminishing the size of the minimal cut-set, however, this will not increase the order of the scramble since Figure \ref{6x6,2} already avoids such eggs. Including any other eggs of size $\leq4$ would decrease the minimal cut set to values less than 27. Thus $sn(K_6\square K_6)=27$. 

We note that $K_6\square K_6$ is the smallest rook graph whose scramble number is strictly greater than the order of $S^*_{n,m}$. Theorem \ref{scramble} shows that for $n\leq 5$, $sn(K_n\square K_m)=||S^*_{n,m}||$ for all but finitely many values of $m$. These cases can be checked by hand.
\end{example}

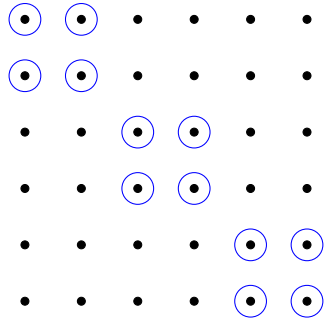
\begin{figure}[ht]
\begin{tikzpicture}[scale=0.75]
\filldraw (0,0) circle (2pt);
\filldraw (0,1) circle (2pt);
\filldraw (0,2) circle (2pt);
\filldraw (0,3) circle (2pt);
\filldraw (0,4) circle (2pt);
\filldraw (0,5) circle (2pt);
\filldraw (1,0) circle (2pt);
\filldraw (1,1) circle (2pt);
\filldraw (1,2) circle (2pt);
\filldraw (1,3) circle (2pt);
\filldraw (1,4) circle (2pt);
\filldraw (1,5) circle (2pt);
\filldraw (2,0) circle (2pt);
\filldraw (2,1) circle (2pt);
\filldraw (2,2) circle (2pt);
\filldraw (2,3) circle (2pt);
\filldraw (2,4) circle (2pt);
\filldraw (2,5) circle (2pt);
\filldraw (3,0) circle (2pt);
\filldraw (3,1) circle (2pt);
\filldraw (3,2) circle (2pt);
\filldraw (3,3) circle (2pt);
\filldraw (3,4) circle (2pt);
\filldraw (3,5) circle (2pt);
\filldraw (4,0) circle (2pt);
\filldraw (4,1) circle (2pt);
\filldraw (4,2) circle (2pt);
\filldraw (4,3) circle (2pt);
\filldraw (4,4) circle (2pt);
\filldraw (4,5) circle (2pt);
\filldraw (5,0) circle (2pt);
\filldraw (5,1) circle (2pt);
\filldraw (5,2) circle (2pt);
\filldraw (5,3) circle (2pt);
\filldraw (5,4) circle (2pt);
\filldraw (5,5) circle (2pt);
\draw [blue](0,5) circle (8pt);
\draw [blue](1,5) circle (8pt);
\draw [blue](0,4) circle (8pt);
\draw [blue](1,4) circle (8pt);
\draw [blue](2,3) circle (8pt);
\draw [blue](3,3) circle (8pt);
\draw [blue](2,2) circle (8pt);
\draw [blue](3,2) circle (8pt);
\draw [blue](4,1) circle (8pt);
\draw [blue](5,1) circle (8pt);
\draw [blue](4,0) circle (8pt);
\draw [blue](5,0) circle (8pt);

\end{tikzpicture}
\caption{A maximal avoidance set of size 12 of the scramble $S^*_{6,6}$}
\label{6x6}
\end{figure}

\begin{figure}[ht]
\begin{tikzpicture}[scale=0.75]
\filldraw (0,0) circle (2pt);
\filldraw (0,1) circle (2pt);
\filldraw (0,2) circle (2pt);
\filldraw (0,3) circle (2pt);
\filldraw (0,4) circle (2pt);
\filldraw (0,5) circle (2pt);
\filldraw (1,0) circle (2pt);
\filldraw (1,1) circle (2pt);
\filldraw (1,2) circle (2pt);
\filldraw (1,3) circle (2pt);
\filldraw (1,4) circle (2pt);
\filldraw (1,5) circle (2pt);
\filldraw (2,0) circle (2pt);
\filldraw (2,1) circle (2pt);
\filldraw (2,2) circle (2pt);
\filldraw (2,3) circle (2pt);
\filldraw (2,4) circle (2pt);
\filldraw (2,5) circle (2pt);
\filldraw (3,0) circle (2pt);
\filldraw (3,1) circle (2pt);
\filldraw (3,2) circle (2pt);
\filldraw (3,3) circle (2pt);
\filldraw (3,4) circle (2pt);
\filldraw (3,5) circle (2pt);
\filldraw (4,0) circle (2pt);
\filldraw (4,1) circle (2pt);
\filldraw (4,2) circle (2pt);
\filldraw (4,3) circle (2pt);
\filldraw (4,4) circle (2pt);
\filldraw (4,5) circle (2pt);
\filldraw (5,0) circle (2pt);
\filldraw (5,1) circle (2pt);
\filldraw (5,2) circle (2pt);
\filldraw (5,3) circle (2pt);
\filldraw (5,4) circle (2pt);
\filldraw (5,5) circle (2pt);
\draw [blue](0,5) circle (8pt);
\draw [blue](1,5) circle (8pt);
\draw [blue](1,4) circle (8pt);
\draw [blue](2,3) circle (8pt);
\draw [blue](3,3) circle (8pt);
\draw [blue](3,2) circle (8pt);
\draw [blue](4,1) circle (8pt);
\draw [blue](5,1) circle (8pt);
\draw [blue](5,0) circle (8pt);

\end{tikzpicture}
\caption{A maximal avoidance set of size 9 of the scramble $T^*$} 
\label{6x6,2}
\end{figure}
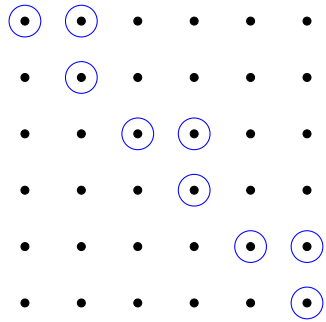

\section{Three Dimensional Rook Graphs} 
Up until this point we have only been concerned with the Cartesian product of two complete graphs. However, if we remember that rook graphs are dual graphs to certain degenerations of complete intersection curves, one might hope to complete the discrete analogue of Lazarsfeld's result on complete intersection curves. Namely, we should expect the following conjecture to be true.
\begin{conjecture}
If $n_1\leq n_2\leq,\dots,\leq n_k$, then gon$(K_{n_1}\square K_{n_2}\square\cdots\square K_{n_k})=(n_1-1)n_2\cdots n_k$. 
\end{conjecture}
This expected gonality is in fact an upper bound since the divisor on $K_{n_1}\square K_{n_2}\square\cdots\square K_{n_k}$ in which one copy of $K_{n_2}\square\cdots\square K_{n_k}$ has no chips and all other vertices have 1 chip, has rank 1. It is rather difficult to show that there is no rank 1 divisor with smaller degree. Unfortunately, it is not clear that the proof of Theorem \ref{gonality} can be generalized for higher dimensions. However, there are some cases where we can use the scramble number to compute the gonality of some three dimensional rook graphs.
\begin{customthm}{1.3}
Let $2\leq n\leq m$, then $sn(K_2\square K_n\square K_m)=$ gon$(K_2\square K_n\square K_m)=nm$
\end{customthm}
\begin{proof}
Consider the scramble $S$ with an egg set consisting of all connected subsets of $n$ vertices. It suffices to show that $||S||=nm$.

We first establish that given any cut $A\coprod B$ such that $|A|, |B|\geq n$ will have $|E(A,B)|\geq nm$. Much of this argument is similar to the proof of Theorem \ref{cutset}. We begin by considering the case where either $A$ or $B$ has at most $m$ vertices. Assume without loss of generality that $n\leq |A|\leq m$. Every vertex has degree $n+m-1$, and thus 
$$|E(A,B)|=(n+m-1)|A|-2k,$$
where $k$ is the number of edges with both ends in $A$. The value $k$ reaches a maximum value when all vertices of $A$ are in a single copy of $K_m$. In this case, we would have $$|E(A,B)|= |A|(m-|A|)+|A|n,$$ 
where $|A|(m-|A|)$ represents the cut edges contained within the single copy of $K_m$, and $|A|n$ represents the number of cut edges between two vertices contained in the same copies of $K_n$ or $K_2$.  
$$|A|(m-|A|)+|A|n= |A|(m+n-|A|).$$
This product is minimized at the boundary cases when $|A|=n$ or $|A|=m$, and the resulting product is $nm$.

Next, we show that the proposition holds if there is at least one copy of $K_2\square K_n$ with all vertices in $A$ and another copy of $K_2\square K_n$ with all vertices in $B$. Without loss of generality, say the first copy of $K_2\square K_n$ is in $A$ and the second is in $B$. Between these two copies of $K_2\square K_n$ there will be $2n$ cut edges. Additionally, for every vertex in the remaining $m-2$ copies of $K-2\square K_n$ there will be a cut edge between that vertex and one of the first two copies of $K_2\square K_n$. Thus we have a minimum of $2n(m-1)=2nm-2n$ cut edges. Since $2\leq n\leq m,$ we have $2n\leq nm$, and thus $$2nm-2n\geq 2nm-nm=nm$$

Finally, we assume there are no copies of $K_2\square K_n$ that only have vertices in $B$ and there are $i$ copies of $K_2\square K_n$ that contain only vertices in $A$, where $0\leq i\leq m-1$. because of our first case, we can assume $|B|\geq m+1$ and therefore there are at least $i(m+1)$ cut edges between the vertices in $B$ and the $i$ copies of $K_2\square K_n$ in $A$. for each of the remaining $m-1$ copies of $K_2\square K_n$ that contain at least one vertex in $A$ and $B$, by Theorem \ref{cutset}, there are at least $n$ cut edges contained in that copy. Therefore 
$$|E(A,B)|\geq n(m-i)+ i(m+1)= nm +i(m+1)-in=nm+i(m+1-n).$$
Since $n\leq m$, we know $(m+1-n)$ is a positive integer and therefore $|E(A,B)|\geq nm.$ This establishes the minimum egg cut set of the scramble $S$ to be $nm$.

Next, one needs to show that the minimum hitting set of $K_2\square K_n\square K_m$ is at least $nm$. To do this we assume that $A\subset V(K_2\square K_n\square K_m)$ only contains $nm-1$ vertices, and then show that it misses some egg in $S$. 

Select the copy of $K_2\square K_n$ that contains the fewest vertices in $A$. This copy must contain fewer than $n$ vertices in $A$ so by the pigeon hole principle, there is some copy of $K_2$ in which neither vertex is contained in $A$. Therefore all vertices in this copy of $K_2\square K_n$ which are not in $A$ are connected. Thus, there are at least $n+1$ connected vertices not contained in $A$, meaning $A$ cannot be a hitting set of $S$. Therefore $||S||=nm$

\end{proof}
While it is certainly possible there are other three dimensional rook graphs which have scramble number equal to the expected gonality, we also have the following result.

\begin{theorem}
If $n\geq 3$, then $sn(K_n\square K_n\square K_n)$ is strictly less than $(n-1)n^2.$
\end{theorem}

\begin{proof}
Consider the following set $A\subset V(K_n\square K_n\square K_n)$:
$$A=\{(1,1,k)|2\leq k\leq n\}\cup \{(1,j,1)|2\leq j\leq n\}\cup \{(i,k,k)|2\leq i\leq n, 1\leq k\leq n\}$$

The set $A$ contains $n+2$ different connected components, each component having $n-1$ vertices. Since $$|A|=(n+2)(n-1)=n^2+n-2\geq n^2+1,$$ we have $|A^c|\leq (n-1)n^2-1$. Furthermore, $|A^c|$ will intersect with every connected vertex subset of size at least $n$. If $S$ is a scramble on $K_n\square K_n\square K_n$, then either, $A^c$ is a hitting set of $S$, or $S$ has an egg contained in one of the connected components of $A$. If the former is true then $||S||\leq |A^c|\leq (n-1)n^2-1$. If the latter is true, then $S$ has an egg $E$, such that $E$ has $n-1$ or fewer vertices, and every vertex in $E$ is adjacent to one another. If $E$ has $i$ vertices, then the egg cut set $E(E,E^c)$ will consist of $i(n-i)+i(2n-2)$ edges. This is because $i(n-i)$ counts the number of edges between the $i$ vertices in $E$ and the $n-i$ vertices not in $E$, but in the same line as $E$. Each vertex in $E$ is also adjacent to $2n-2$ other vertices which are not in the same line as $E$. Then we use the fact that $n\geq 3$ and $i\leq n-1$ to get $$i(n-i)+i(2n-2)=i(3n-i-2)\leq i(n^2-i-2)\leq (n-1)(n^2-i-2)<(n-i)n^2.$$
Thus $S$ will have either a hitting set or an egg cut set smaller than $(n-1)n^2$. 
\end{proof}
\begin{figure}
\begin{tikzpicture}[scale=1.4]
\filldraw (0,0,0) circle (2pt);
\filldraw [blue] (0,1,0) circle (4pt);
\filldraw [blue] (0,2,0) circle (4pt);
\filldraw (1,0,0) circle (2pt);
\filldraw (1,1,0) circle (2pt);
\filldraw (1,2,0) circle (2pt);
\filldraw (2,0,0) circle (2pt);
\filldraw (2,1,0) circle (2pt);
\filldraw (2,2,0) circle (2pt);
\filldraw [blue] (0,0,1) circle (4pt);
\filldraw (0,1,1) circle (2pt);
\filldraw (0,2,1) circle (2pt);
\filldraw (1,0,1) circle (2pt);
\filldraw [blue] (1,1,1) circle (4pt);
\filldraw (1,2,1) circle (2pt);
\filldraw (2,0,1) circle (2pt);
\filldraw [blue] (2,1,1) circle (4pt);
\filldraw (2,2,1) circle (2pt);
\filldraw [blue] (0,0,2) circle (4pt);
\filldraw (0,1,2) circle (2pt);
\filldraw (0,2,2) circle (2pt);
\filldraw (1,0,2) circle (2pt);
\filldraw (1,1,2) circle (2pt);
\filldraw [blue] (1,2,2) circle (4pt);
\filldraw (2,0,2) circle (2pt);
\filldraw (2,1,2) circle (2pt);
\filldraw [blue] (2,2,2) circle (4pt);
\draw [dashed] (0,0,0)--(2,0,0);
\draw [dashed] (0,0,0)--(0,2,0);
\draw [dashed] (0,0,0)--(0,0,2);
\draw (2,0,0)--(2,2,0);
\draw (2,0,0)--(2,0,2);
\draw (0,2,0)--(0,2,2);
\draw (0,2,2)--(2,2,2);
\draw (2,0,2)--(2,2,2);
\draw (2,2,0)--(2,2,2);
\draw (0,0,2)--(0,2,2);
\draw (0,0,2)--(2,0,2);
\draw (0,2,0)--(2,2,0);
\draw (0,2,1)--(2,2,1);
\draw (1,2,2)--(1,2,0);
\draw (0,1,2)--(2,1,2);
\draw [dashed] (0,1,1)--(2,1,1);
\draw [dashed] (0,1,0)--(2,1,0);
\draw [dashed] (0,1,0)--(0,1,2);
\draw [dashed] (1,1,0)--(1,1,2);
\draw (2,1,0)--(2,1,2);
\draw [dashed] (0,0,1)--(0,2,1);
\draw [dashed] (0,0,1)--(2,0,1);
\draw [dashed] (1,0,0)--(1,0,2);
\draw (2,0,1)--(2,2,1);
\draw (1,0,2)--(1,2,2);
\draw [dashed] (1,0,1)--(1,2,1);
\draw [dashed] (1,0,0)--(1,2,0);

\end{tikzpicture}
\caption{The set $A$ for  $K_3\square K_3\square K_3$ as described in the proof of Theorem 6.3}

\end{figure}
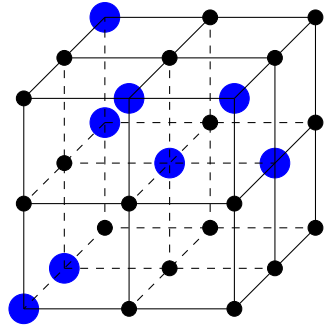
\section{Higher Gonalities of Two Dimensional Rook Graphs}

Recall that the gonality of a graph $G$ is the fewest number of chips needed to construct a divisor of rank 1. In other words, if gon$(G)=j$ then there is some divisor $D$ of degree $j$ such that if someone were to take one chip from any vertex, we could get back to an effective divisor through a series of chip fires. We can expand on this idea by asking how many chips we need to make a divisor that can withstand the theft of more than just one chip.

\begin{definition}
Given a graph $G$, the $k-$\emph{gonality} of $G$, denoted gon$_k(G)$ is the fewest number of chips needed to construct a divisor of rank $k$.
\end{definition}

\begin{lemma}
Given a graph $G$, gon$_k(G)\leq$ gon$_{k+1}(G)-1$.
\end{lemma}

\begin{proof}
Let $D$ be a divisor on $G$ of degree gon$_{k+1}(G)$, and having rank $k+1$. If we take away a single chip from any vertex which has a positive number of chips, we are left with a divisor of degree gon$_{k+1}(G)-1$ and rank $k$. Therefore gon$_k(G)\leq$ gon$_{k+1}(G)-1$.
\end{proof}

With this lemma, we have enough information to modify the proof of \ref{gonality} in order to compute the 2 and 3 gonalities of two dimensional rook graphs
\begin{customthm}{1.4}
Given the rook graph $K_n\square K_m$, gon$_2(K_n\square K_m)=nm-1$ and gon$_3(K_n\square K_m)=nm$.
\end{customthm}

\begin{proof}
There are two main components to this proof. First, we will show that there exists a divisor of degree $nm$ of rank at least 3. Second, we will show that no divisor of degree $nm-2$ has rank 2. 

Given a graph $G$, one can show that the divisor $D$ has positive rank if for every vertex $v\in V(G)$, there is some equivalent effective divisor $D'$, where $v$ has a positive number of chips. Similarly, $D$ has rank at least $k$ if we can take away any $k-1$ chips and still be left with a divisor such that for any vertex $v$, there is some equivalent effective divisor where $v$ has a positive number of chips. Let $D$ be the divisor on $K_n\square K_m$ in which every vertex has exactly 1 chip. We then consider all possible ways to take away 2 chips from $D$.
\begin{itemize}
 
\item If 2 chips are taken from a single vertex $v$, we can fire all vertices except $v$. Then $v$ will have $n+m-3$ chips and all other vertices have either 1 or 0 chips. Assuming $n,m\geq 2$, $v$ will have a positive number of chips.

\item If chips are taken from $v_1$ and $v_2$ which lie in the same row or column, we can fire all of the other rows or columns to produce an effective divisor where both $v_1$ and $v_2$ have a positive number of chips.

\item If chips are taken from $v_1$ and $v_2$ which lie in different rows and columns, we can fire all vertices except $v_1$ to produce an effective divisor where $v_1$ has a positive number of chips. Similarly, we can fire all vertices except $v_2$ to produce a different effective divisor where $v_2$ has a positive number of chips.

\end{itemize} 
Therefore $D$ has rank at least 3 and gon$_3(K_n\square k_m)\leq nm$.

Next, we let $E$ be an effective divisor of degree $nm-2$. We assume $E$ has a maximal number of chips in its poorest column. First, we consider the case where the poorest columns have $n-1$ chips. We begin by selecting one of the poorest columns, and removing a chip from a vertex that has a positive number of chips. By Lemma \ref{poorest row}, starting a fire on any vertex in this column with no chips will result in the entire column burning. Now that an entire column is burning, every other column with n-1 chips must also burn due to the same argument as in the proof of Lemma \ref{poorest row}. We can understand this by thinking of the original burning column as a singular burning vertex that is adjacent to every vertex in a given column. This then reduces to the case where we have the graph $K_{n+1}$ with only $n-1$ chips.

Since $E$ is degree $nm-2$, and the poorest column has $n-1$ chips, if a column has $n+i$ chips for $i\geq 0$, then there must be at least $i+2$ other columns with $n-1$ chips. All of these $i+2$ columns with $n-1$ chips will burn, which will lead to the column with $n+i$ chips burning as well. Therefore the entire graph burns and $E$ does not have rank 2. 

Next, we assume that $E$ has a poorest column with $\leq n-2$ chips. We remove a chip from a poorest row, which would have at most $m-1$ chips. We then start a fire on some vertex with zero chips in that row. This row now has at most $m-2$ chips, and thus it must burn entirely. Any column with $\leq n-2$ chips will also burn. We then assume for contradiction that some subset $U\subset V(K_n\square K_m)$ will be left unburnt. If we fire the vertices in $U$ to produce the new divisor $E'$, we know from the proof of Theorem \ref{gonality} that any column which intersects with $U$ will have at least $n-1$ chips in $E'$. Every other column will increase the number of chips it has by $|U|$. This contradicts the fact that $E$ maximized the chips in the poorest column, and the poorest column had $\leq n-2$ chips. Therefore the entire graph must burn and $E$ does not have rank 2. Thus we conclude that gon$_2(K_n\square K_m)=nm-1$ and gon$_3(K_n\square K_m)=nm$.
\end{proof}
\clearpage
\bibliographystyle{alpha}
\bibliography{ref}

\begin{thebibliography}{vDdBG20}

\bibitem[AM20]{Ralph}
Ivan Aidun and Ralph Morrison.
\newblock On the gonality of {C}artesian products of graphs.
\newblock {\em Electron. J. Combin.}, 27(4):Paper No. 4.52, 35, 2020.

\bibitem[BN07]{BakerNorine}
Matthew Baker and Serguei Norine.
\newblock Riemann-{R}och and {A}bel-{J}acobi theory on a finite graph.
\newblock {\em Adv. Math.}, 215(2):766--788, 2007.

\bibitem[Dha90]{DBA}
Deepak Dhar.
\newblock Self-organized critical state of sandpile automaton models.
\newblock {\em Phys. Rev. Lett.}, 64(14):1613--1616, 1990.

\bibitem[HJJS22]{paper1}
Michael Harp, Elijah Jackson, David Jensen, and Noah Speeter.
\newblock A new lower bound on graph gonality.
\newblock {\em Discrete Applied Mathematics}, 309:172--179, 2022.

\bibitem[Laz97]{lazarsfeld1994lectures}
Robert Lazarsfeld.
\newblock Lectures on linear series.
\newblock In {\em Complex algebraic geometry ({P}ark {C}ity, {UT}, 1993)},
  volume~3 of {\em IAS/Park City Math. Ser.}, pages 161--219. Amer. Math. Soc.,
  Providence, RI, 1997.
\newblock With the assistance of Guillermo Fern\'{a}ndez del Busto.

\bibitem[vDdBG20]{TW}
Josse van Dobben~de Bruyn and Dion Gijswijt.
\newblock Treewidth is a lower bound on graph gonality.
\newblock {\em Algebr. Comb.}, 3(4):941--953, 2020.

\end{thebibliography}
\end{document}